\documentclass[11pt]{article}
\usepackage{a4wide,amsmath,amssymb}
\usepackage[toc,page]{appendix}
\usepackage{hyperref}
\usepackage{url}
\usepackage{courier}
\usepackage{easyReview}
\usepackage{amsmath,amssymb,amsthm}
\usepackage{mathtools}
\usepackage{xcolor}

\newtheorem{thm}{Theorem}[section]

\newtheorem{prop}[thm]{Proposition}

\theoremstyle{remark}
\newtheorem{rem}[thm]{Remark}

\theoremstyle{definition}

\newtheoremstyle{Claim}{}{}{\itshape}{}{\itshape\bfseries}{:}{ }{#1}
\theoremstyle{Claim}

\newcommand{\R}{\mathbb{R}}

\newcommand{\mI}{\mathcal I}
\newcommand{\mG}{\mathcal G}

\newcommand{\half}{\frac 1 2}
\newcommand{\pd}{\partial}

\theoremstyle{plain}

\def\sideremark#1{\ifvmode\leavevmode\fi\vadjust{
		\vbox to0pt{\hbox to 0pt{\hskip\hsize\hskip1em
				\vbox{\hsize3cm\tiny\raggedright\pretolerance10000
					\noindent #1\hfill}\hss}\vbox to8pt{\vfil}\vss}}}

\makeatletter
\@namedef{subjclassname@2020}{\textup{2020} Mathematics Subject Classification}
\makeatother

	\title{Li-Yau  inequality and related
		properties on metric star graphs}

	\author{Fabio Camilli\thanks{Dip. di Ingegneria e Geologia, Universit\`a degli Studi ``G. d'Annunzio" Chieti-Pescara, viale Pindaro 42, 65127 Pescara (Italy) ({\tt fabio.camilli@unich.it}).}}

\begin{document}
\maketitle	
	
	\begin{abstract} 
		We prove a Li-Yau gradient estimate for positive solutions to the heat equation defined on a metric star graph $\mG$ given by the heat kernel formula. As consequence, we derive a Harnack estimate and a Liouville property for bounded  harmonic functions. The argument exploits an explicit representation formula for the heat kernel on $\mG$.
	\end{abstract}



\section{Introduction}
In \cite{ly},   Li and   Yau proved the classic gradient estimates (the Li-Yau inequality)
  \begin{equation}\label{LY}
	\Delta(\ln(u))\ge -\frac{n}{2t}
\end{equation} 
for  a positive solution $u$  of the heat equation on a  $n$-dimensional compact Riemannian manifold with non-negative Ricci curvature where $\Delta$ is the Laplace-Beltrami operator. As applications of \eqref{LY}, the Harnack inequality and the  Liouville property for harmonic functions can be derived.  Many efforts have been made to establish analogous results in different settings and for more general equations (\cite{bl,cm,ham,sz}).  
The previous theory has been expanded  in \cite{bhllmy} (see also \cite{glly,q}) to encompass the case of a connected, finite graph with  the Laplacian   substituted with a finite difference operator known as the graph Laplacian. \par
An intermediary between a   manifold and a  graph is represented by a metric graph, or network,   comprising a   set of vertices interconnected by continuous, non-self-intersecting edges. In this setting, the heat equation  is defined on each edge, with appropriate transition conditions of Kirchhoff type prescribed at the vertices (see \cite{kps,mu}). 
Although there are different proofs of \eqref{LY}, the classical argument is based on deriving   twice the  equation satisfied by $\ln(u)$  and using the maximum principle for the resulting nonlinear pde. It does not seem to be possible to use this approach for the case of networks.  Indeed, by deriving the heat equation, the transition conditions at the vertices are lost and the resulting problem no longer satisfies a maximum principle.\par
For the particular case of metric star graphs, i.e. a vertex  connected with   $N$ half-lines, we obtain an analogous of the inequality \eqref{LY} for   solutions of the heat equation  given by the heat kernel representation formula   in \cite{o}.   The resulting inequality contains an additional term,   vanishing at the vertex, which takes into account the non-symmetry of the problem.  We also discuss two classical consequences of the Li-Yau inequality, respectively a Harnack inequality and a Liouville property for bounded harmonic functions on   metric star graphs.

\subsection{Notations:} A metric star graph $\mathcal G$  is composed of a vertex $O$ to which are connected $N$ edges $e_1,\dots,e_N$,  each edge being a half-line. 
In the following, we  identify each edge $e_i$, $i=1,\dots,N$, with the interval $I_i=[0,\infty)$, with $0$ corresponding to the vertex $O$. 
A schematic representation of a star graph is given in Figure \ref{fig:star_graph}. 
\begin{figure}[htbp!]\label{fig:star_graph}
	\centering
	\begin{tikzpicture}[scale=0.040]
		\node at (0,0) {$\bullet$};
	 	\node at (8,0) {$O$};
	 	\node at (36,18) {$I_i$};
		\draw[rotate=40] (0,0) -- (40,0) node[anchor=west]  {};
		\draw[rotate=110] (0,0) -- (40,0) node[anchor=west] {};
		\draw[rotate=180] (0,0) -- (40,0)node[anchor=north] {};
		\draw[rotate=240] (0,0) -- (40,0)node[anchor=north] {};
		\draw[rotate=320] (0,0) -- (40,0)node[anchor=north] {};
	\end{tikzpicture}    
	\caption{A star graph with $5$ edges.}
\end{figure}
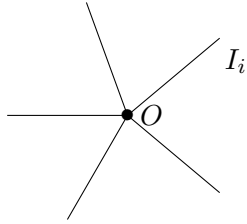
We consider the distance function on $\mG$   given by
\begin{equation}\label{eq:dist}
	\mathrm{dist}(x,y)=\begin{cases}
		|x-y|& x,y\in I_i,\, i=1,\dots,N\\
		x+y& x\in I_i,\, y\in I_j,\,i\neq j ,
	\end{cases}
\end{equation}
and we set $B_R=\{x\in\mathcal G: \, \mathrm{dist}(x,0)\le R\}$.
A function $u$ on $\mathcal G$ is a collection of functions  $u_j:I_j\to\mathbb \R$, $j=1,\dots,N$ and we set $\int_{\mathcal G}u dx =\sum_{j=1}^N\int_0^\infty u_j(x)dx$.\\
We consider the heat equation on $\mG$
\begin{equation}\label{eq:heat}
	\left\{
	\begin{array}{ll}
	\pd_t u_i-\pd_{xx}u_i=0\qquad&(x,t)\in I_i\times (0,\infty),\, i\in \{1,\dots,N\},\\[6pt]
	 u_i(0,t)=u_j(0,t)& t\in(0,\infty),\,i, j\in \{1,\dots,N\},\\[6pt]
	 \sum_{j=1}^N\alpha_j \pd_x u_j(0,t)=0& t\in(0,\infty),\\[6pt]
	 u_i(x,0)=\phi_i(x)&x\in I_i,\, i\in \{1,\dots,N\},
	\end{array}
	\right.
\end{equation} 
where $\alpha_i>0$, $\sum_{i=1}^N\alpha_i=1$ and $\phi:{\mathcal G}\to \R$  with $\phi\ge 0$, $\int_{\mathcal G}\phi(x) dx<+\infty$.
\begin{rem}
In \eqref{eq:heat}, the heat equation is defined inside each edge $I_i$, while at vertex $O$ we require continuity of the solution and the classical Kirchhoff condition. Well posedness of \eqref{eq:heat} is discussed in \cite{kps,mu}, while its probabilistic interpretation in \cite{fs}.
\end{rem}
Let 
\begin{equation}\label{g}
	g(z,t)=\frac{1}{(4\pi t)^\half}e^{-\frac{|z|^2}{4t}}\qquad z\in\R,\, t\in (0,\infty),
\end{equation}
be the standard heat kernel. Then, the formula
\begin{equation}\label{eq:sol_heat}
	P_t\phi(x)=\int_{\mathcal G}\Gamma_i(x,y,t)\phi(y)dy=\sum_{j=1}^N\int_0^{\infty}\Gamma_i(x,y,t)\phi_j(y)dy\qquad (x,t)\in I_i\times\R^+,
\end{equation}
where
\begin{equation}\label{eq:kernel}
	\Gamma_i(x,y,t)=
	\begin{cases}
		2\alpha_j g(x+y,t)&x\in I_i, y\in I_j,\, j\neq i,\\[4pt]
		 g(x-y,t)+(2\alpha_i-1)g(x+y,t) &x,\,y\in I_i,
	\end{cases}
\end{equation}
gives a solution of  \eqref{eq:heat}  (see \cite{fs,o}).
\begin{rem}
Formulas for the heat kernel     has been computed in \cite{cc,kps}  for  the case of  finite metric graph and in \cite{c}  for an infinite homogeneous tree   and   a countable
metric graph.	
\end{rem}
\section{The Li-Yau inequality}
In this section, we prove a Li-Yau inequality (see \cite{ly})  for the heat equation defined on  the metric star graph $\mG$. We will consider solutions of the heat equation given by
\eqref{eq:sol_heat} and we will exploit properties of the heat kernel $g(z,t)$ in \eqref{g}.
\begin{thm}\label{th:liyau}
Let $u=P_t\phi$. Then
	\begin{equation}\label{eq:liyau_est}
	\pd_{xx}(\ln(u))\ge -\frac{1}{2t}-(1-2\alpha_i)\mI_i(x,t), \qquad x\in I_i, \,t>0,
	\end{equation}
where
\begin{equation}\label{eq:I}
	\mI_i(x,t)=\frac{1}{4t^2P_t\phi(x)}\int_0^\infty (|x+y|-|x-y|)^2\frac{g(x-y,t)g(x+y,t)}{g(x-y,t)+(2\alpha_i-1)g(x+y,t)}\phi_i(y)dy.
\end{equation}
\end{thm}
\begin{proof} 
Set  $v(x,t)=\ln (P_t\phi(x))$, where $x\in  I_i$ and $P_t$ is defined as in \eqref{eq:sol_heat}.  Since
 \[\pd_t g(z,t)= \frac{|z|^2}{4t^2}g(z,t) -\frac{1}{2t}g(z,t), \]
we have 
\begin{equation}\label{eq:ly2}
	\begin{split}
		\pd_t v(x,t)&= \frac{ \pd_t P_t\phi(x)}{P_t \phi(x) }  \\
		&=\frac{ 1}{P_t \phi(x) } \Big[\sum_{j\neq i }2\alpha_j \int_0^\infty\big(\frac{|x+y|^2}{4t^2}-\frac{1}{2t}\big)g(x+y,t)\phi_j(y)dy\\
		&+\int_0^\infty \big(\frac{|x-y|^2}{4t^2}g(x-y,t)+(2\alpha_i-1)\frac{|x+y|^2}{4t^2}g(x+y,t)\big)\phi_i(y)dy\\
		&-\int_0^\infty\frac{1}{2t} \big(g(x-y,t)+(2\alpha_i-1)g(x+y,t)\big)\phi_i(y)dy\Big]\\
		&=\frac{ 1}{P_t \phi(x) } \Big[\sum_{j\neq i }2\alpha_j \int_0^\infty\frac{|x+y|^2}{4t^2}g(x+y,t)\phi_j(y)dy\\
		&+\int_0^\infty \big(\frac{|x-y|^2}{4t^2}g(x-y,t)+(2\alpha_i-1)\frac{|x+y|^2}{4t^2}g(x+y,t)\big)\phi_i(y)dy\Big]-\frac{1}{2t}.
	\end{split}
\end{equation}
Moreover, since
\[\pd_x v(x,t)=  \frac{ \pd_x P_t\phi(x)}{P_t \phi(x) }=\frac{1}{P_t \phi(x)}\int_{\mG}\pd_x\Gamma_i(x,y,t)\phi(y)dy,\]
by H\"older's inequality we get
\begin{equation}\label{eq:ly1}
	\begin{split}
	|\pd_x v(x,t)|&\le \frac{1}{P_t \phi(x)} \int_\mG\frac{|\pd_x\Gamma_i(x,y,t)|}{\sqrt{\Gamma_i(x,y,t)}}\sqrt{\phi(y)}\sqrt{\Gamma_i(x,y,t)}\sqrt{\phi(y)}dy\\
	&\le  \frac{1}{P_t \phi(x)}\Big(\int_\mG\frac{|\pd_x\Gamma_i(x,y,t)|^2}{ \Gamma_i(x,y,t)}\phi(y)dy\Big)^\half\Big(\int_\mG \Gamma_i(x,y,t) \phi(y)dy\Big)^\half\\
	&=\frac{1}{\sqrt{P_t \phi(x)}} \Big(\sum_{j=1}^N\int_0^{\infty}\frac{|\pd_x\Gamma_i(x,y,t)|^2}{ \Gamma_i(x,y,t)}\phi_j(y)dy\Big)^\half.
\end{split}
\end{equation}
Since $|\pd_z g(x,t)|=|z|g(z,t)/2t$, we have
\begin{equation}\label{eq:ly3}
	\frac{|\pd_x\Gamma_i(x,y,t)|^2}{ \Gamma_i(x,y,t)}=
	\begin{cases}
		2\alpha_j \dfrac{|x+y|^2}{4t^2}g(x+y,t)\quad&y\in I_j,\, j\neq i,\\[8pt]
		\dfrac{1}{4t^2}\frac{\big(|x-y|g(x-y,t)+(2\alpha_i-1)|x+y|g(x+y,t)\big)^2}{g(x-y,t)+(2\alpha_i-1)g(x+y,t)}&y\in I_i,
	\end{cases}
\end{equation}
and we observe that 
\begin{equation}\label{eq:ly4}
\begin{split}
&\frac{1}{4t^2}\frac{\big(|x-y| g(x-y,t)+(2\alpha_i-1)|x+y|g(x+y,t)\big)^2}{g(x-y,t)+(2\alpha_i-1)g(x+y,t)}\\
&=\frac{|x-y|^2}{4t^2}g(x-y,t)+(2 \alpha_i-1)\frac{|x+y|^2}{4t^2}	g(x+y,t) \\
&- \frac{(2\alpha_i-1)}{4t^2}\big[(|x+y|-|x-y|)^2\frac{g(x-y,t)g(x+y,t)}{g(x-y,t)+(2\alpha_i-1)g(x+y,t)}
\big].
\end{split}
\end{equation}
Exploiting \eqref{eq:ly2}, \eqref{eq:ly3} and \eqref{eq:ly4} in \eqref{eq:ly1} and recalling \eqref{eq:I}, we get
\begin{equation}\label{eq:ly5}
	\begin{split}
	|\pd_x v(x,t)|^2&\le \frac{1}{P_t \phi(x)}  \sum_{j=1}^N\int_0^{\infty}\frac{|\pd_x\Gamma_i(x,y,t)|^2}{ \Gamma_i(x,y,t)}\phi_j(y)dy \\
	&  =\frac{1}{P_t \phi(x)}\Big[\sum_{j\neq i}2\alpha_j\int_0^\infty \dfrac{|x+y|^2}{4t^2}g(x+y,t)\phi_j(y)dy\\
&	+ \int_0^\infty\Big(\frac{|x-y|^2}{4t^2}g(x-y,t)+(2 \alpha_i-1)\frac{|x+y|^2}{4t^2}	g(x+y,t)\Big)\phi_i(y)dy\Big]\\
&+(1-2\alpha_i)\mI_i(x,t)=\frac{\pd_t P_t\phi(x)}{P_t\phi(x)}+\frac{1}{2t}+(1-2\alpha_i)\mI(x,t)\\
&=\pd_t (v(x,t))+\frac{1}{2t}+(1-2\alpha_i)\mI_i(x,t).
\end{split}
\end{equation}
Hence 
\[
\pd_{xx}(\ln(u(x,t))=\pd_t (v(x,t))-|\pd_x v(x,t)|^2\ge -\frac{1}{2t}-(1-2\alpha_i)\mI_i(x,t)
\]
and therefore \eqref{eq:liyau_est}.
\end{proof}

\begin{rem}
Since  $g(x-y,t)-g(x+y,t)+2\alpha_i g(x+y,t)>0$ for $x,y\in\R^+$, then   $\mI_i(x,t)>0$ for $x\in I_i$, $t>0$. 
Computing $(|x+y|-|x-y|)^2$ for $x,y\ge 0$, we get
\begin{equation}\label{eq:I_bis}
	\begin{split}
	\mI_i(x,t)&=\frac{1}{t^2P_t\phi(x)}\big(\int_0^x y^2 h(x,y,t) \phi_i(y)dy+\int_x^\infty x^2 h(x,y,t) \phi_i(y)dy\big)\\
	&\le \frac{x^2}{t^2P_t\phi(x)} \int_0^\infty  h(x,y,t) \phi_i(y)dy
\end{split}
	\end{equation}
where
\[h(x,y,t)=\frac{g(x-y,t)g(x+y,t)}{g(x-y,t)+(2\alpha_i-1)g(x+y,t)}.\]
Since $g(x-y,t)-g(x+y,t)> 0$, we have $h(x,y,t)\le g(x-y,t)/2\alpha_i$ and therefore, by \eqref{eq:I_bis}, we get
\begin{equation}\label{eq:stima_I}
\mI_i(x,t)\le \frac{x^2}{t^2}\frac{\int_0^\infty  g(x-y,t) \phi_i(y)dy}{2\alpha_i P_t\phi(x)}
\end{equation}
and 
\begin{equation}\label{eq:stima_I_bis}
\mI_i(x,t)	\le \frac{x^2}{t^2}\frac{ (4\pi t)^{-\half} \|\phi_i\|_{L^1(\R^+)}}{2\alpha_i P_t\phi(x)}.
\end{equation}

\end{rem}

\begin{rem}

In the Euclidean setting, the Li-Yau estimate is known to be optimal because the heat kernel satisfies \eqref{LY} with equality. Interestingly, the real line can be regarded as a special case of a metric graph \(\mG = I_1 \cup I_2\) with \(\alpha_1 = \alpha_2 = 1/2\). In this context, we have \(\Gamma_i(x, y, t) = g(x - y, t)\) and \(\mI_i(x, t) \equiv 0\) for \(i = 1, 2\). Consequently, the solution of the heat equation given by \eqref{eq:sol_heat} for an initial condition \(\phi\) defined as a Dirac   \(\delta_O\) at the origin becomes \(u(x, t) = g(x, t)\) for \((x, t) \in I_j \times \mathbb{R}^+\), \(j = 1, 2\), achieving equality in \eqref{eq:liyau_est}.

In general, the relationship between the additional term \(\mI_i\), the initial condition \(\phi\), and the coefficients \(\alpha_j\) remains unclear. To explore this further, we compute two explicit examples:

\emph{ Example 1:} Consider a metric graph with three edges \(e_j\), where \(\alpha_j = 1/3\), and an initial condition \(\phi\) defined as \(\phi_j(y) = \delta_1(y)\) for \(j = 1, 2, 3\). Using \eqref{eq:sol_heat}, we find  
\[
P_t\phi(x) = g(x + 1, t) + g(x - 1, t) \qquad (x, t) \in I_j \times \mathbb{R}^+, \, j = 1, 2, 3.
\]  
For \(u(x, t) = P_t\phi(x)\), a straightforward computation gives:  
\[
\partial_{xx} \ln(u(x, t)) = -\frac{1}{2t} + \frac{1}{t^2} \frac{e^{-\frac{x}{t}}}{\left(1 + e^{-\frac{x}{t}}\right)^2}.
\]  
In this case, the estimate \eqref{eq:liyau_est} is not optimal. However, the classical Li-Yau estimate is recovered in the limit as \(x \to +\infty\).

\emph{Example 2:} Consider the  metric graph and the coefficients \(\alpha_j\) as in Example 1, but with an initial condition \(\phi\) such that \(\phi_j(y) = 0\) for \(j = 1, 2\), and \(\phi_3(y) = \delta_1(y)\). Here, the solution \(u(x, t) = P_t\phi(x)\) becomes  
\[
u(x, t) =
\begin{cases}
	\frac{2}{3} g(x + 1, t) & (x, t) \in I_j \times \mathbb{R}^+, \, j = 1, 2, \\
	g(x - 1, t) - \frac{1}{3} g(x + 1, t) & (x, t) \in I_3 \times \mathbb{R}^+.
\end{cases}
\]  
The second derivative of the logarithm of $u$ is computed as:  
\[
\partial_{xx} \ln(u(x, t)) =
\begin{cases}
	-\frac{1}{2t} & (x, t) \in I_j \times \mathbb{R}^+, \, j = 1, 2, \\[2pt]
	-\frac{1}{2t} - \frac{1}{t^2} \frac{\frac{1}{3}e^{-\frac{x}{t}}}{\left(1 - \frac{1}{3} e^{-\frac{x}{t}}\right)^2} & (x, t) \in I_3 \times \mathbb{R}^+.
\end{cases}
\]  
While the classical Li-Yau estimate holds with equality on \(e_1\) and \(e_2\), an additional negative term appears on \(e_3\).

These computations highlight that the interpretation of the additional term \(\mI_i\) remains an open problem, which merits further investigation.
\end{rem}

\section{Applications}
In this section we discuss some consequence of the estimate  proved in the previous section. As a classical application of the Li-Yau inequality \eqref{eq:liyau_est}, we prove
a Harnack estimate for the  solution of the heat equation.
\begin{prop}
	Let $u=P_t\phi$ with $\phi>0$ on $\mG$. Given    $R>0$ and $\epsilon>0$, we have for $x,y\in B_R$, $s,t\in (\epsilon,+\infty)$, $s<t$,
\begin{equation}\label{eq:harnack_est}
\begin{split}
	\frac{u(x,t)}{u(y,s)}\ge  \left(\frac{s}{t}\right)^\half\exp\Big(-(\frac{1}{4} + C)\frac{d(x,y)^2}{t-s}- C\frac{t-s}{st}  \rho(x,y,t,s) ^2 \\
	-2 C\frac{d(x,y)}{t-s}\rho(x,y,t,s)\ln\left(\frac{s}{t}\right)\Big),
\end{split}	
\end{equation}
where $C$ depends on $R$, $\epsilon$ and $\phi$, and
\[
\rho(x,y,t,s)=\begin{cases}
x-\frac{d(x,y)}{t-s}t&x\in I_i, y\in I_j,\, j\neq i\\
x+\frac{d(x,y)}{t-s}t &x,\,y\in I_i.
\end{cases}
\]
\end{prop}
\begin{proof}
We first consider the case $x,y\in I_i$, $i\in\{1,\dots, N\}$. Set
	\[w(r)=v(x+r\frac{y-x}{t-s},t-r)\qquad r\in [0,t-s]\]
where $v(x,t)=\ln(u(x,t))$. Observe that $w(0)=v(x,t)$, $w(r)=v(y,s)$ and
\begin{align*}
	w'(r) &=\pd_x v \Big(x+r\frac{y-x}{t-s},t-r\Big)\frac{y-x}{t-s}-\pd_t v \Big(x+r\frac{y-x}{t-s},t-r\Big) \\
	&\le |\pd_xv|^2 +\frac{1}{4}\left|\frac{y-x}{t-s}\right|^2 -\pd_t v \Big(x+r\frac{y-x}{t-s},t-r\Big).
\end{align*}
Hence, by \eqref{eq:liyau_est} and \eqref{eq:stima_I}, we get
\begin{equation}\label{stima_deriv}
	w'(r)\le \frac{1}{2(t-r)}+\frac{1}{4}\left|\frac{y-x}{t-s}\right|^2+ C
\frac{(x+r (y-x)/(t-s))^2}{(t-r)^2}
\end{equation}
where 
\begin{equation}\label{constant}
	C=\max_{i=1,\dots, N}\frac{\int_0^\infty  g(x+r\frac{y-x}{t-s}-z,t-r) \phi_i(z)dy}{2\alpha_i P_{t-r}\phi(x)}.
\end{equation}
We prove that the constant $C$ is bounded. We have
\begin{equation}\label{parte_1}
	\int_0^\infty  g(x+r\frac{y-x}{t-s}-z,t-r) \phi_i(z)dz\le (4\pi (t-r))^{-\half}  \|\phi_i\|_{L^1(\R^+)}.
\end{equation}  
Moreover, if $\phi$ is not identically null, there exists $\bar j\in \{1,\dots, N\}$, $\eta>0$, $R_0>0$ and a measurable set $K\subset B_{R_0}$ with positive Lebesgue measure $\lambda$ such that $\phi_{\bar j}(x)\ge \eta$ for $x\in K$. Hence for $x,y\in B_R$, since $g(z_1-z_2,t)-g(z_1+z_2,t)\ge 0$ for $z_1+z_2\ge 0$, we have
\begin{equation}\label{parte_2}
	\begin{split}
	P_{t-r}\phi(x)&\ge  \sum_{j=1}^N 2\alpha_j\int_0^\infty g(x+r\frac{y-x}{t-s}+z,t-r)\phi_j(z)dz\\
	&\ge 2\alpha_{\bar j}\int_0^\infty g(x+r\frac{y-x}{t-s}+z,t-r)\phi_{\bar j}(z)dz\\
	&\ge 2\alpha_{\bar j} \eta \int_K g(x+r\frac{y-x}{t-s}+z,t-r)dz\\
	&\ge 2\alpha_{\bar j} (4\pi (t-r))^{-\half} \eta\lambda e^{-\frac{|R_0+R|^2}{4\epsilon}}.
	\end{split}
\end{equation}
Plugging \eqref{parte_1} and \eqref{parte_2} in \eqref{constant}, we get a bound on $C$ 
which depends on $R$, $\epsilon$ and $\phi$.\\
By \eqref{stima_deriv}, we have 
\begin{align*}
	w(r)-w(0) &\le \int_0^{t-s} \left[\frac{1}{2(t-r)}+\frac{1}{4}\left|\frac{y-x}{t-s}\right|^2+ C
	\frac{(x+r (y-x)/(t-s))^2}{(t-r)^2}\right]dr\\
	&=\half\ln\left(\frac{t}{s}\right)+(\frac{1}{4} + C)\frac{|y-x|^2}{t-s} +
	 C\Big (x+\frac{y-x}{t-s}t\Big)^2 \frac{t-s}{st}\\
	&+2 C\frac{ y-x }{t-s}\Big (x+\frac{y-x}{t-s}t\Big)\ln\left(\frac{s}{t}\right)
\end{align*}
and, recalling the definition of $w$ and $v$, we get
\begin{equation}\label{eq:hk1}
	\begin{split}
	\frac{u(x,t)}{u(y,s)}\ge  \left(\frac{s}{t}\right)^\half\exp\Big(-(\frac{1}{4} + C)\frac{|y-x|^2}{t-s}- C\frac{t-s}{st}\Big (x+\frac{y-x}{t-s}t\Big)^2 \\
	-2 C\frac{|y-x|}{t-s}\Big (x+\frac{y-x}{t-s}t\Big)\ln\left(\frac{s}{t}\right)\Big).
\end{split}
\end{equation}
Now we consider the case $x\in I_i$, $y\in I_j$, $i\neq j$. Recall that $\mathrm{dist}(x,y)=x+y$. Define
\[
\gamma(r)=\begin{cases}
	x-\frac{x+y}{t-s}r\quad &r\in [0, \frac{x}{x+y}(t-s)]\\[4pt]
	 \frac{x+y}{t-s}r-x &r\in [ \frac{x}{x+y}(t-s),t-s]
\end{cases}
\]
and $w(r)=v(\gamma(r),t-r)$. Then $w(0)=v(x,t)$, $w(t-s)=v(y,s)$ and 
 
\begin{align*}
	w'(r)&=\pd_x v(\gamma(r),t-r)\dot \gamma(r)-\pd_t v(\gamma(r),t-r)\\
	&=
\begin{cases}
	-\pd_x v(\gamma(r),t-r)\frac{x+y}{t-s} -\pd_t v(\gamma(r),t-r) \quad &r\in [0, \frac{x}{x+y}(t-s)],\\[4pt]
	\pd_x v(\gamma(r),t-r)\frac{x+y}{t-s} -\pd_t v(\gamma(r),t-r)   &r\in [ \frac{x}{x+y}(t-s),t-s].
\end{cases}   
\end{align*}
Hence,  by \eqref{eq:liyau_est} and \eqref{eq:stima_I}, we get
\[w'(r)\le \frac{1}{2(t-r)}+\frac{1}{4}\left|\frac{x+y}{t-s}\right|^2+ C
\frac{(\gamma(r))^2}{(t-r)^2}\qquad r\in [0,t-s],\]
with $(\gamma(r))^2=(x- r(x+y)/(t-s))^2$. Arguing as in the previous case we get
\begin{equation}\label{eq:hk2}
	\begin{split}
		\frac{u(x,t)}{u(y,s)}\ge  \left(\frac{s}{t}\right)^\half\exp\Big(-(\frac{1}{4} + C)\frac{(x+y)^2}{t-s}- C\frac{t-s}{st} \Big (x-\frac{x+y}{t-s}t\Big)^2 \\
		-2 C\frac{x+y}{t-s}\Big (x-\frac{x+y}{t-s}t\Big)\ln\left(\frac{s}{t}\right)\Big).
	\end{split}
\end{equation}
Combining \eqref{eq:hk1} with \eqref{eq:hk2}, the desired result \eqref{eq:harnack_est} follows.
\end{proof}
 We say that a function $w$ is harmonic on $\mG$  if it satisfies
\begin{equation}\label{eq:laplace}
	\left\{
	\begin{array}{ll}
		 \pd_{xx}w_i=0\quad&x\in I_i ,\, i\in \{1,\dots,N\},\\[4pt]
		w_i(0 )=w_j(0 )&  i, j\in \{1,\dots,N\},\\[4pt]
		\sum_{j=1}^N\alpha_j w_j(0 )=0.&  
	\end{array}
	\right.
\end{equation} 
By means of  inequality \eqref{eq:liyau_est}, we   prove  a classical Liouville property for harmonic functions on metric star graph.
\begin{prop}
	Any bounded  harmonic function $w$ on $\mG$ is constant.
\end{prop}
\begin{proof}
	Let $w$ be a solution of \eqref{eq:laplace}. Since $ w$ is bounded, we can add a constant to ensure $ w$ is strictly positive. Now, consider the heat equation \eqref{eq:heat} with the initial condition $ w$. Given that $ w$ is bounded, the problem admits the unique solution $u(x, t)=w(x)$, represented by \eqref{eq:sol_heat}.  
	By \eqref{eq:liyau_est}, since $\pd_t (\ln u )=0$, we have
	\[
 |\pd_x \ln(u(x,t))|^2 \le \frac{1}{2t}-(1-2\alpha_i)\mI_i(x,t).
	\]
Recalling \eqref{eq:stima_I} and passing to the limit for $t\to\infty$ in the previous
inequality, we get $\pd_x \ln(w(x))=0$ and therefore $w$ constant on $\mG$.
\end{proof}
\section*{Acknowledgments} The author expresses gratitude to Alessandro Goffi for  insightful discussions on the subject of this work. 


\end{document}